\documentclass[letterpaper, 10 pt, conference]{IEEEtran}
\usepackage[all]{xy}
\usepackage{color}
\usepackage{algorithm}
\usepackage{algpseudocode}
\usepackage{subfigure}
\usepackage{bm}
\usepackage{graphicx}
\usepackage{color}
\usepackage{wrapfig}
\usepackage{epsf}
\usepackage{times}
\usepackage{url}
\usepackage{amssymb, amsmath, mathtools, amsthm}
\usepackage{nomencl}
\usepackage{comment}

\usepackage{enumitem}
\makenomenclature
\newtheorem{theorem}{Theorem}

\newtheorem{corollary}{Corollary}

\definecolor{RED}{rgb}{0.6,0.,0.}
\definecolor{BLUE}{rgb}{0.,0.,0.6}
\definecolor{GREEN}{rgb}{0.,0.6,0.}
\definecolor{MALINA}{rgb}{0.6,0.,0.6}
\definecolor{YELLOW}{rgb}{0.8,0.8,0}
\newcommand{\squeezeup}{\vspace{-1.5 mm}}
\begin{document}
\title{Topology Estimation in Bulk Power Grids:\\ Guarantees on Exact Recovery}
\author{
\IEEEauthorblockN{Deepjyoti~Deka\dag, Saurav~Talukdar\ddag, Michael~Chertkov\dag, and Murti~Salapaka \ddag}\\
\IEEEauthorblockA{(\dag) Los Alamos National Laboratory, New Mexico, USA,
(\ddag)University of Minnesota Twin Cities, Minneapolis, USA}
}

\maketitle
\begin{abstract}
The topology of a power grid affects its dynamic operation and settlement in the electricity market. Real-time topology identification can enable faster control action following an emergency scenario like failure of a line. This article discusses a graphical model framework for topology estimation in bulk power grids (both loopy transmission and radial distribution) using measurements of voltage collected from the grid nodes. The graphical model for the probability distribution of nodal voltages in linear power flow models is shown to include additional edges along with the operational edges in the true grid. Our proposed estimation algorithms first learn the graphical model and subsequently extract the operational edges using either thresholding or a neighborhood counting scheme. For grid topologies containing no three-node cycles (two buses do not share a common neighbor), we prove that an exact extraction of the operational topology is theoretically guaranteed. This includes a majority of distribution grids that have radial topologies. For grids that include cycles of length three, we provide sufficient conditions that ensure existence of algorithms for exact reconstruction. In particular, for grids with constant impedance per unit length and uniform injection covariances, this observation leads to conditions on geographical placement of the buses. The performance of algorithms is demonstrated in test case simulations.
\end{abstract}

\section{Introduction}
\label{sec:intro}
The power grid comprises of the set of transmission lines that transfer power from generators to the end users. Structurally, the grid is described as a graph with nodes representing buses and edges representing the transmission lines. It is worth noting that high voltage transmission grids are topologically loopy (with cycles) while medium voltage distribution grids are generally radial (no cycles) in structure \cite{hoffman2006practical}. The true operational topology in either case is determined by the current breaker/switch statuses on an underlying set of permissible edges as shown in Fig.~\ref{fig:city}. Topology estimation/learning refers to the problem of determining the current operational topology and is a necessary part for majority of control and optimization problems in the dynamic and static regimes of grid operation. Real-time topology estimation can enable timely detection of line failures and identification of critical lines that affect locational marginal prices. Such estimation is hampered by the limited presence of real-time line-based measurements especially in medium and low voltage part of the grid. In recent years, there has been a surge in installation of nodal/bus based measurement devices like phasor measurement units (PMUs) \cite{PMU}, micro-PMUs \cite{micropmu}, FNETs \cite{FNET}, smart sensors that record high-fidelity real-time measurements at nodes/buses (not lines) to enhance observability and then use the information to control, e.g. smart devices. We analyze the problem of real-time topology estimation using voltage measurements collected from meters located at the grid nodes through the framework of graphical models.
\subsection{Prior Work}
Research in topology estimation in the power grid has explored different directions to utilize the available measurements under varying operating regimes. Primarily such research has focused on distribution grids that are operationally radial in topology and use measurements from static power flow models. \cite{bolognani2013identification} presents a topology identification algorithm for radial grids with constant $r/x$ (resistance to reactance) ratio for transmission lines using signs of elements in the inverse covariance matrix of voltage magnitudes. In a similar radial setting, \cite{dekapscc} presents the use of conditional independence based tests to identify the radial topology from voltage measurements for general distributions of nodal injections. Greedy schemes to identify radial topologies based on provable trends in second moments of voltage magnitudes are presented in \cite{dekatcns}. This has been extended to cases with missing nodes in \cite{dekaecc, dekasmartgridcomm}. Signature/envelope based identification of topology changes is proposed in \cite{berkeley,arya}. In \cite{sharon2012topology}, a machine learning (ML) topology estimate with approximate schemes is used in a distribution grid with Gaussian loads. For loopy power grids, approximate schemes for topology estimation have been discussed in \cite{ram_loop} but do not have guarantees on exact recovery. For measurements from grid dynamics, \cite{sauravacc, sauravacm} discuss the use of Wiener filters for topology estimation in radial and loopy grids.

\subsection{Contribution}
In this work, we present a graphical model \cite{wainwright2008graphical} based learning scheme for topology estimation in general power grids, both loopy and radial, using linear power flow models. Our approach factorizes the empirical distribution of nodal voltages collected by the meters and then extracts the true topology (operational edges) from the estimated graphical model. We show that for general grids, the graphical model of the distribution of nodal voltage measurements is a super-set of the original topology. We present algorithms and conditions under which the true grid topology can be extracted from the estimated graphical model. For radial grids, we show that a local neighborhood counting scheme or thresholding in graphical model concentration matrix are capable of recovering the true topology, irrespective of the line impedance values. For grids with cycles, the thresholding scheme is guaranteed to estimate the true topology when minimum cycle length is greater than three (no triangular cycles), while the neighborhood counting method is able to estimate for cases with minimum cycle lengths greater than six. Finally for grids with triangles, we provide sufficient conditions under which thresholding based tests are able to recover the true topology. These conditions depend on the line impedances and injection fluctuations at the grid nodes. In particular, for grids with constant impedance per unit length and similar injection fluctuations, we present sufficiency conditions for exact recovery that depend on geometry of the grid layout.

The next section presents nomenclature and power flow relations in bulk power grids. Section \ref{sec:graphicalmodel} develops the graphical model of power grid voltage measurements under two linear power flow models and discusses key aspects of its structure. Section \ref{sec:graphicalmodel} includes a discussion of efficient schemes to estimate the graphical model. The first topology learning using neighborhood counting is presented in Section \ref{sec:neigh} along with with conditions for exact recovery. Section \ref{sec:thres} describes the second learning algorithm that includes thresholding of values in the inverse covariance matrix of nodal voltages. Section \ref{sec:simulations} includes simulations results of our work on IEEE test cases. Conclusions and future work are included in Section \ref{sec:conclusion}.

\section{Power Grid: Structure and Power Flows}
\label{sec:structure}
\textbf{Structure}: We represent the topology of the power grid by the graph ${\cal G}=({\cal V},{\cal E})$, where ${\cal V}$ is the set of $N+1$ buses/nodes of the graph and ${\cal E}$ is the set of undirected lines/edges that represent the operational lines. This operational grid is determined by closed switches/breakers in an underlying set of permissible lines ${\cal E}^{full}$ (see Fig.~\ref{fig:city}). For distribution grids, the operational edge set $\cal E$ represents a tree (radial network) while for transmission grid the operational network is loopy in general. The operational lines are unknown to the observer who may or may not have access to the permissible edge set ${\cal E}^{full}$. As ${\cal E}^{full}$ may not be known, we consider all node pairs as permissible edges. We denote nodes by alphabets $i$, $j$ and so on. The undirected edge connecting nodes $i$ and $j$ is denoted by $(ij)$. Let ${\cal P}^{j}_{i} \equiv \{(ik_1), (k_1k_2),...(k_{n-1}j)\}$ denote a set of $n$ distinct undirected edges in $\cal G$ that connect node $i$ and node $j$. We call ${\cal P}^{j}_{i}$ as a path of length $n$ from $i$ to $j$. If $i = j$ then path ${\cal P}^{j}_{i}$ is termed a `cycle'. By definition, a cycle has length three or more. Note that for radial networks, there is a unique path between each distinct pair of nodes, while for loopy networks there can be multiple paths between a node pair. An example of a path is shown in Fig.~\ref{fig:city}. The length of shortest path between two nodes is called the `separation' between them. The set of nodes that share an edge with node $i$ are called its neighbors, while the set of nodes that are separated by two hops from node $i$ are called its two-hop neighbors. Nodes with degree one are termed `leaves'. The single neighbor of a leaf is called its `parent'.
\begin{figure}[!bt]
\centering
\includegraphics[width=0.35\textwidth, height =.26\textwidth]{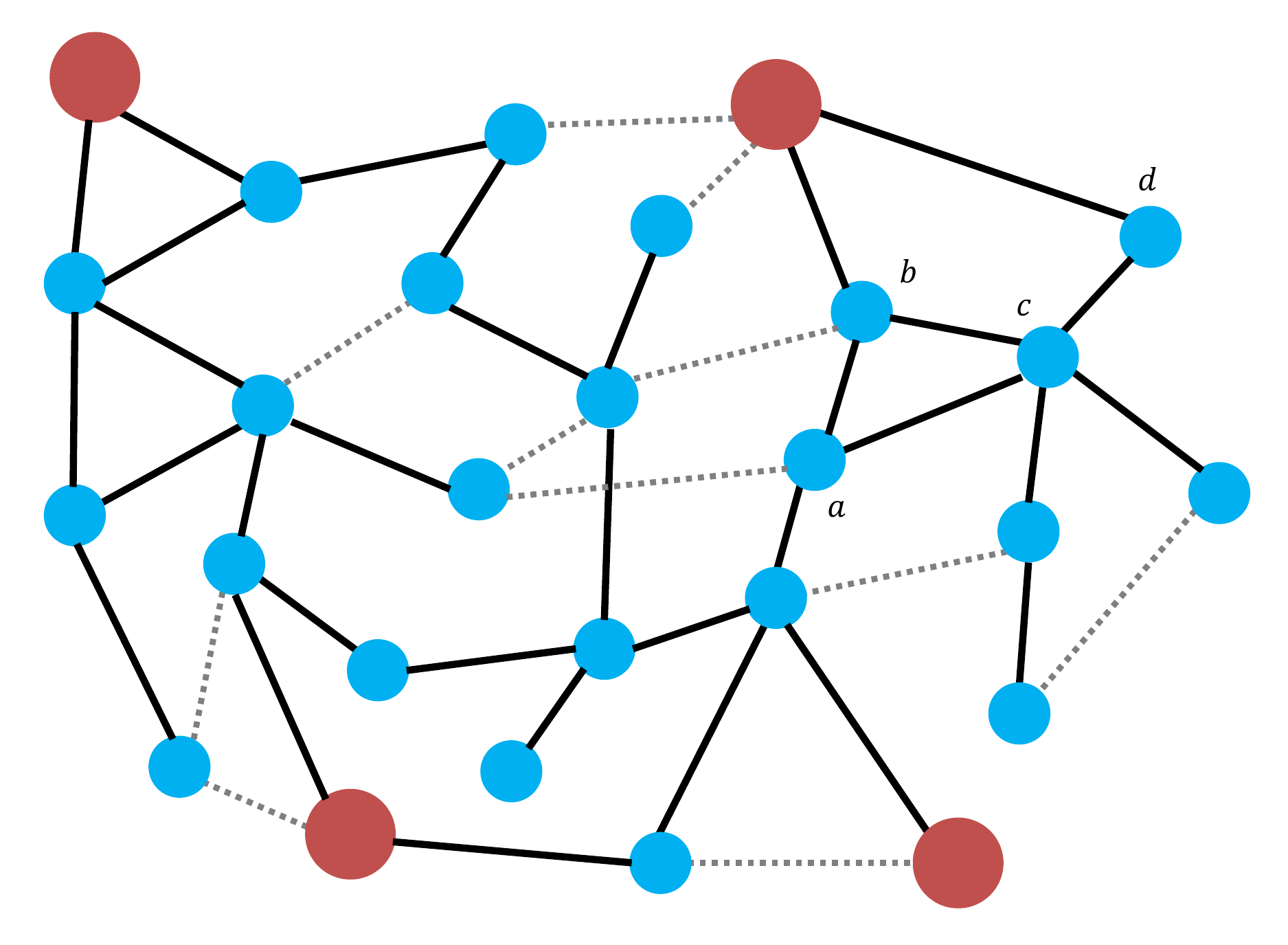}
\caption{Example of a power grid with $4$ substations. Substations are represented by large red nodes. The operational grid is formed by solid lines (black). Dotted grey lines represent open switches. ${\cal P}_a^b \equiv \{(ab), (bc), (cd)\}$ represents a non-unique path from $a$ to $d$.
\label{fig:city}}
\end{figure}
Next we discuss the power flow models in the grid.

\textbf{Power Flow (PF) Models}: Let $z_{ij}=r_{ij}+\hat{i} x_{ij}$ denote the complex impedances of line $(ij)$ in the grid ($\hat{i}^2=-1$) with resistance $r_{ij}$ and reactance $x_{ij}$. By Kirchhoff's laws, the complex valued PF equation for the flow of power out of a node $i$ in grid ${\cal G}$ is given by,
\begin{align}
 P_i =p_i+\hat{i} q_i&=\underset{j:(ij)\in{\cal E}}{\sum} V_i(V_i^* - V_j^*)/z_{ij}^*\label{P-complex},\\
&= \underset{j:(ij)\in{\cal E}}{\sum}\frac{v_i^2-v_i v_j\exp(\hat{i}\theta_i-\hat{i}\theta_j)}{z_{ij}^*},\label{P-complex1}
\end{align}
where, the real valued scalars, $v_i$, $\theta_i$, $p_i$ and $q_i$ denote the voltage magnitude, voltage phase, active and reactive power injections respectively at node $i$. $V_i (= v_i\exp(\hat{i}\theta_i))$ and $P_i$ denote the nodal complex voltage and injection respectively. During normal operation in bulk power systems, one can assume that the grid is lossless and the net power (generation minus demand) in the grid is zero. One bus then is considered as reference bus, with its power injection equal to negative of the sum of injections at all other buses. Further the voltage and phase at all other buses are measured relative to the respective values at the reference bus which has $0$ phase and $1$ p.u. voltage magnitude. Without a loss of generality, we ignore the reference node/bus and restrict the power flow analysis to the $N$ non-reference buses in the grid. Under the loss-less assumption, we use the following well-known relaxations to the PF equations.

\textbf{DC Power Flow (DC-PF) model} \cite{abur2004power}: Here all lines in the grid are considered to be purely susceptive, voltage magnitudes are assumed to be constant at unity and phase differences between neighboring lines are assumed to be small. This leads to the following linear relation between active power flows and phase angles,
\begin{align}
\forall i\in {\cal V}: &p_i = \sum_{j:(ij)\in{\cal E}} \beta_{ij}(\theta_i-\theta_j), \nonumber\\
\text{which, in vector form is,~} &p = H_{\beta}\theta. \label{DC_PF}
\end{align}
Here, $H_{\beta}$ is given by,
\begin{align}
 H_{\beta}(i,j) = \begin{cases}\sum_{k:(ik) \in {\cal E}}\beta_{ik} ~~\text{if~} i=j\\
-\beta_{ij} ~~\text{if~} (ij) \in {\cal E}\\
 0 ~~\text{otherwise}\end{cases}.\label{Laplacian}
\end{align}
Thus, $H_{\beta}$ is the reduced weighted Laplacian matrix for the grid ${\cal G}$ with edge weights given by susceptances $\beta$ ($\beta_{ij}=\frac{x_{ij}}{x_{ij}^2+r_{ij}^2}$ for edge $(ij)$). The reduction is derived by removing the row and column corresponding to the reference bus from the weighted Laplacian matrix.

\textbf{Linear Coupled Power Flow (LC-PF) model} \cite{dekatcns,bolognani2016existence}: The a.c. power flow equations in (\ref{P-complex}) are linearized assuming small deviations in both phase difference of neighboring nodes ($|\theta_i - \theta_j| <<1$ for edge $(ij)$), and voltage magnitude deviations from the reference bus ($|v -1|<<1$), which leads to,
\begin{align}
&\begin{bmatrix}v \\\theta\end{bmatrix}= \begin{bmatrix}H_g &H_{\beta}\\ H_{\beta} &-H_{g}\end{bmatrix}^{-1}\begin{bmatrix}p \\q\end{bmatrix}, \label{PF_LPV_p}
\end{align}
where
For radial grids, it can be verified that
\begin{align}
\begin{bmatrix}H_g &H_{\beta}\\ H_{\beta} &-H_{g}\end{bmatrix}^{-1} = \begin{bmatrix}H^{-1}_{1/r} &H^{-1}_{1/x}\\ H^{-1}_{1/x} &-H^{-1}_{1/r}\end{bmatrix}\nonumber
\end{align}
where $H_{1/r}$ ($H_{1/x}$) is the reduced weighted Laplacian with the weight associated with each edge $(ij)$ being $1/r_{ij}$ ($1/x_{ij}$). This makes Eq.~(\ref{PF_LPV_p}) equivalent to the LinDistFlow model \cite{89BWa,89BWb,89BWc} used in radial networks. Thus the LC-PF equations are a generalization of the LinDistFlow equations to general loopy grids \cite{dekatcns}. Note that both the DC and LC models represent an invertible map between nodal injections and voltages at the non-substation buses. This property is used in determining the graphical model of nodal voltages discussed next.
\squeezeup
\section{Graphical Model of Voltages}
\label{sec:graphicalmodel}
Graphical models represent the structure within a multi-variate probability distribution. In this section, we develop the graphical model for the distribution of nodal voltages in the linear power flow models discussed (DC-PF and LC-PF). First, we make the following assumption regarding the injections at the non-reference nodes in the grid.

\textbf{Assumption $1$}: Injection fluctuations at non-reference nodes in the grid are independent zero-mean Gaussian random variables with non-zero covariances.

Note that over short to medium intervals, fluctuations in grid nodes are due to changes in loads or noise that can be assumed to be independent and uncorrelated. The mean of the fluctuations can be empirically estimated and de-trended and hence ignored without a loss of generality. In fact the Gaussian assumption is not necessary for majority of our analysis but is taken as the most accepted model of disturbance/noise.

For the DC-PF, the vector of active injections at all nodes $p$ in grid $\cal G$ is a Gaussian random variable ${\cal P}^{DC}(p)\equiv \mathcal{N}(0, \Sigma^{DC}_p)$ where $\Sigma^{DC}_p$ is the diagonal covariance matrix of active injections at the non-reference buses. On the other hand, in the LC-PF model, the injection vector \text{\footnotesize$\begin{bmatrix}p\\q\end{bmatrix}$} comprises of active and reactive injections and is modelled by the Gaussian random variable ${\cal P}^{LC}(p, q)\equiv \mathcal{N}(0, \Sigma^{LC}_{(p,q)})$. Here
\begin{align}
\squeezeup
\Sigma^{LC}_{(p,q)} = \begin{bmatrix}\Sigma^{LC}_{pp} &\Sigma^{LC}_{pq} \\ \Sigma^{LC}_{qp} &\Sigma^{LC}_{qq}\end{bmatrix}\label{sigma_lc}
\end{align}
is the covariance matrix between active and reactive injections at the node. Each block in $\Sigma^{LC}_{(p,q)}$ is a diagonal matrix. Note that as active and reactive injections at each node may be correlated, $\Sigma^{LC}_{pq} = \Sigma^{LC}_{qp}\neq 0$. 

We now derive the probability distribution for nodal phase angles ${\cal P}^{DC}(\theta)$ in the DC-PF as,
\begin{align}
{\cal P}^{DC}(\theta) &= \frac{1}{|J_P^{DC}(\theta)|}{\cal P}^{DC}( H_{\beta}\theta)\nonumber,\\
&= \frac{1}{|J_P^{DC}(\theta)|}\prod_{i\in{\cal V}}{\cal P}^{DC}_i\left(\smashoperator[r]{\sum_{j:(ij)\in{\cal E}}} \beta_{ij}(\theta_i-\theta_j)\right) \label{GM_V}
\end{align}
Here $J^{DC}_P(\theta)$ is the constant Jacobian involved in the linear transformation from phase angles to injections. Note that the probability distribution for injections is in product form as the fluctuations are assumed to be independent and holds true for non-Gaussian distributions as well. For Gaussian injections, we can use the linear relation for DC-PF to derive the distribution for phase angles directly. This follows from the result that a linear function of Gaussian random variables is also Gaussian \cite{gubner2006probability}. The probability distribution of the nodal phase angles is given by:
\begin{align}
{\cal P}^{DC}(\theta)&\equiv\mathcal{N}(0, \Sigma^{DC}_{\theta}) \text{~where~}\Sigma^{DC}_{\theta} = H_{\beta}^{-1}\Sigma^{DC}_p H_{\beta}^{-1}\label{covar_theta}
\end{align}
where Eq.~(\ref{covar_theta}) follows from Eq.~(\ref{DC_PF}). Similarly for the LC-PF model, the distribution of voltage magnitudes and phase angles $(v,\theta)$ is given by a Gaussian random variable with covariance matrix given by:
\begin{align}
\text{\small
$\Sigma^{LC}_{(v,\theta)} = \begin{bmatrix}H_g &H_{\beta}\\ H_{\beta} &-H_{g}\end{bmatrix}^{-1}\begin{bmatrix}\Sigma^{LC}_{pp} &\Sigma^{LC}_{pq} \\ \Sigma^{LC}_{qp} &\Sigma^{LC}_{qq}\end{bmatrix}\begin{bmatrix}H_g &H_{\beta}\\ H_{\beta} &-H_{g}\end{bmatrix}^{-1}$}\label{covar_varepsilon}
\end{align}
In either case, we represent the probability distribution of the nodal voltages using a graphical model. We first formally define a `Graphical Model'.

\textbf{Graphical Model}: For a $n$ dimensional random vector $X = [X_1, X_2,..X_n]^T$, the corresponding undirected graphical model ${\cal GM}$ \cite{wainwright2008graphical} has vertex set ${\cal V_{GM}}$ where each node represents one variable. For every node $i$, its graph neighbors form the smallest set of nodes $N(i) \subset {\cal V_{GM}}- \{i\}$ such that for any node $j \not\in N(i)$, the distribution at $i$ is conditionally independent of $j$ given set $N(i)$, i.e., ${\cal P}(X_i|X_ {N(i)},X_j) = {\cal P}(X_i|X_ {N(i)})$. For Gaussian graphical models, it is known that the edges in the graphical model for a random vector are given by the non-zero terms in the inverse covariance matrix (also called `concentration' matrix) \cite{wainwright2008graphical}. In other words, for node $i$, node $j \in N(i)$ if $\Sigma^{-1}(i,j) \neq 0$.

We first focus on the graphical model for the distribution of phase angles in the DC-PF model. The following theorem gives its graphical model representation.

\begin{theorem}\label{structure_DC_PF}
Consider grid graph $\cal G$ with nodal injection fluctuations modelled as independent Gaussian random variables. The graphical model for nodal phase angles in the DC-PF model consists of edges between nodes representing phase angles at all neighbors and two-hop neighbors in $\cal G$.\end{theorem}
\begin{proof}
To determine the edges in the graphical model, we analyze the inverse covariance matrix of the phase angles. Using Eqs.~(\ref{Laplacian}, \ref{covar_theta}), it is clear that for $i \neq j$\squeezeup
\begin{align}
{\Sigma^{DC}_{\theta}}^{-1}(i,j) = \begin{cases}-\beta_{ij}(\frac{\beta_i}{\Sigma^{DC}_p(i,i)}+\frac{\beta_j}{\Sigma^{DC}_p(j,j)})+ \\\smashoperator[r]{\sum_{k \neq i,j}}\frac{\beta_{ik}\beta_{jk}}{\Sigma^{DC}_p(k,k)} \text{~if edge $(ij)\in {\cal E}$}\\
\smashoperator[r]{\sum_{k \neq i,j}}\frac{\beta_{ik}\beta_{jk}}{\Sigma^{DC}_p(k,k)} \text{~if $i,j$ are two hops away},\\
0 \text{~otherwise} \end{cases}\nonumber 
\end{align}
Here $\beta_i = \sum_{j \neq i} \beta_{ij}$. It is thus clear that the inverse covariance matrix, for each node $i$ has non-zero terms at its neighbors and two-hop neighbors, which define the edges in the graphical model.
\end{proof}
An example of a grid and the associated graphical model for phase angles in DC-PF are given in Fig.~\ref{fig:grid} and \ref{fig:graphical_DC} respectively.
\begin{figure}[tb]
\centering\hfill
\subfigure[]{\includegraphics[width=0.44\columnwidth]{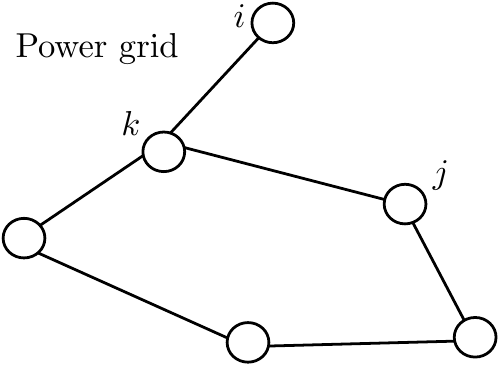}\label{fig:grid}}\hfill
\subfigure[]{\includegraphics[width=0.53\columnwidth]{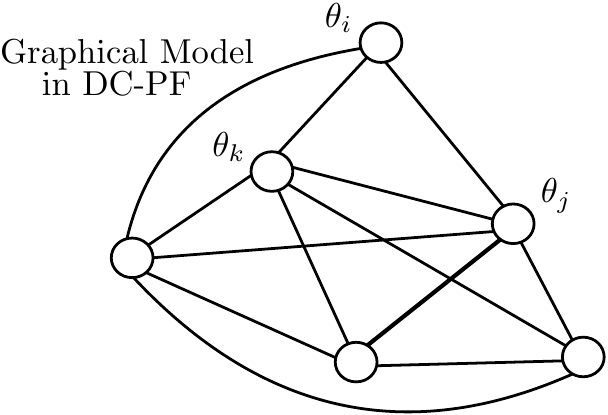}\label{fig:graphical_DC}}\hfill
\subfigure[]{\includegraphics[width=0.53\columnwidth]{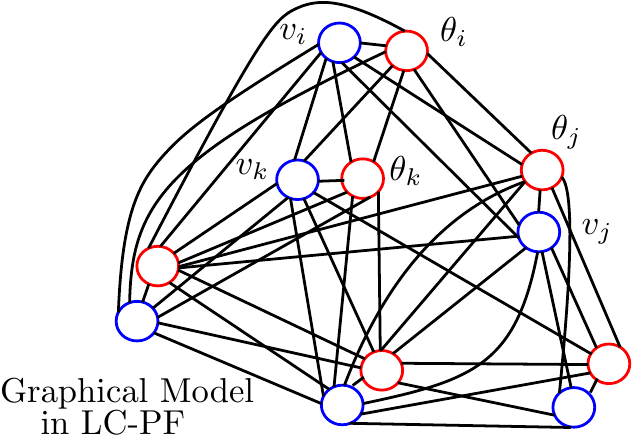}\label{fig:graphical_LC}}\hfill
\subfigure[]{\includegraphics[width=0.45\columnwidth]{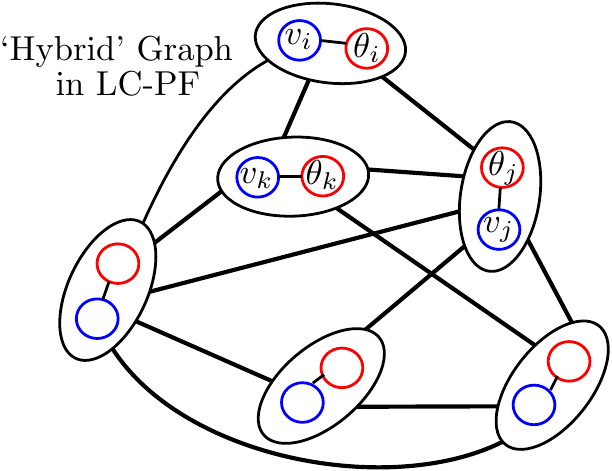}\label{fig:graphical_hybrid}}\hfill
\caption{(a) Power grid graph ${\mathcal{G}}$ (b) Graphical model for phase angles in DC-PF (c) Graphical model for voltage magnitudes and phase angles in LC-PF (d) `Hybrid' graph by combining nodes for same bus in graphical model for LC-PF. \label{fig:graphical_model}}
\end{figure}
Next we look at the graphical model of voltage magnitudes and phase angles in the LC-PF model. As each grid node is described by voltage magnitude and phase angle, the number of nodes in the graphical model is twice the number of grid buses. The following theorem gives the inverse of the covariance matrix $\Sigma^{LC}_{(v,\theta)}$ in LC-PF.

\begin{theorem}\label{inverse_LC_PF}
For LC-PF, the inverse covariance matrix ${\Sigma^{LC}_{(v,\theta)}}^{-1}$ satisfies
\begin{align*}
&{\Sigma^{LC}_{(v,\theta)}}^{-1} = \begin{bmatrix} J_{vv} &J_{v\theta}\\J_{\theta v} &J_{\theta\theta}\end{bmatrix}\text{~where~}\\
&J_{vv}=\text{\small $H_{g}D^{-1}(\Sigma^{LC}_{qq}H_{g}-\Sigma^{LC}_{pq}H_{\beta})-H_{\beta}D^{-1}(\Sigma^{LC}_{pq} H_{g}-\Sigma^{LC}_{pp}H_{\beta})$}\nonumber\\
&J_{v\theta}=\text{\small $H_{\beta}D^{-1}(\Sigma^{LC}_{qq}H_{\beta} +\Sigma^{LC}_{pq}H_{g})-H_{\beta}D^{-1}(\Sigma^{LC}_{pq}H_{\beta}+\Sigma^{LC}_{pp}H_{g})$}\nonumber\\
&J_{\theta v}=\text{\small $H_{\beta}D^{-1}(\Sigma^{LC}_{qq}H_{g} -\Sigma^{LC}_{pq}H_{\beta})+H_{g}D^{-1}(\Sigma^{LC}_{pq} H_{g}-\Sigma^{LC}_{pp}H_{\beta})$}\nonumber\\
&J_{\theta\theta}=\text{\small $H_{\beta}D^{-1}(\Sigma^{LC}_{qq}H_{\beta} +\Sigma^{LC}_{pq}H_{g})+H_{g}D^{-1}(\Sigma^{LC}_{pq}H_{\beta}+\Sigma^{LC}_{pp}H_{g})$}\nonumber
\end{align*}
where matrix $D$ is diagonal with $$D(i,i) = |\Sigma^{LC}_{pp}(i,i)\Sigma^{LC}_{qq}(i,i)-{\Sigma^{LC}_{pq}}^2(i,i)|$$
\end{theorem}
\begin{proof}
We derive the inverse of $\Sigma^{LC}_{(v,\theta)}$ by inverting each matrix on the right side of Eq.~(\ref{covar_varepsilon}). From Assumption 1, it is clear that each block in $\Sigma^{LC}_{(p,q)}$ (see Eq.~(\ref{sigma_lc})) is a diagonal matrix with blocks $\Sigma^{LC}_{pq} = \Sigma^{LC}_{qp}$. Using Schur's compliment expansion or by direct multiplication it can be verified that,
\begin{align}
\begin{bmatrix}\Sigma^{LC}_{pp} &\Sigma^{LC}_{pq} \\ \Sigma^{LC}_{pq} &\Sigma^{LC}_{qq}\end{bmatrix}^{-1} = \begin{bmatrix}D^{-1}\Sigma^{LC}_{qq} &-D^{-1}\Sigma^{LC}_{pq}\\ -D^{-1}\Sigma^{LC}_{pq} &D^{-1}\Sigma^{LC}_{pp}\end{bmatrix},\nonumber
\end{align}
where, $D$ is a diagonal matrix specified in the theorem statement. Note that the $i^{th}$ diagonal entry in $D$ is given by the determinant of $\begin{bmatrix}\Sigma_p(i,i) &\Sigma_{pq}(i,i) \\\Sigma_{pq}(i,i) &\Sigma_q(i,i)\end{bmatrix}$, the injection covariance matrix at node $i$. Multiplying the individual inverses result in the expression given in the statement of the theorem.
\end{proof}
Using the expression for inverse of $\Sigma^{LC}_{(v,\theta)}$, we present the structure of the graphical model of nodal voltage magnitudes and phase angles in the LC-PF model.
\begin{theorem}\label{structure_LC_PF}
Consider grid graph $\cal G$ with nodal injection fluctuations modelled as independent Gaussian random variables. The graphical model for nodal voltage magnitudes and phase angles in the LC-PF model consists of edges between voltage magnitudes and phase angles at the same bus,  at all neighbors, and at two-hop neighbors in $\cal G$.
\end{theorem}
\begin{proof}
Consider four terms in the expression for $J_{v\theta}$ in Theorem \ref{inverse_LC_PF} where $D^{-1}, \Sigma^{LC}_{qq}, \Sigma^{LC}_{pq}, \Sigma^{LC}_{pp}$ are all diagonal matrices. Using analysis presented in Theorem \ref{structure_DC_PF}, we can show that $J_{v\theta}(i,j) \neq 0$ if $i = j$ or $i$ and $j$ correspond to neighbors or two-hop neighbors in $\cal G$. Thus the node representing voltage magnitude at $i$ in the graphical model is linked to nodes for phases at $i$, all neighbors and two-hop neighbors of $i$. A similar argument for $J_{vv}$, $J_{\theta v}$ and $J_{\theta\theta}$ proves the theorem.
\end{proof}
Theorem \ref{structure_LC_PF} thus implies that the graphical model for voltage magnitudes and phase angles in the LC-PF model does not include any edge between voltages corresponding to buses that are three or more hops away in grid $\cal G$. Indeed if we combine the voltage and phase angle at each bus into a `hybrid' node and connect such nodes if edges exist between their constituent voltages and phase angles in the graphical model, we get a `hybrid' graph of $N$ nodes that has the same structure as the graphical model for phase angles in the DC-PF. This is depicted in Figs.~\ref{fig:graphical_LC} and \ref{fig:graphical_hybrid}.

\textbf{Remarks:} A few points are in order.
\begin{itemize}[leftmargin=*]
\item For the general case with non-Gaussian but independent injection fluctuations, the probability distribution for phase angles in DC-PF is given in Eq.~(\ref{GM_V}). It can be shown that the graphical model in this case is also given by the true topological edges in the grid and edges between the two-hop neighbors (see details for a radial case in \cite{dekapscc}).
\item The graphical model for only voltage magnitude deviations (no phase angles) in the LC-PF model has a similar structure as discussed in Theorem \ref{structure_LC_PF}, if the resistance to reactance ratio on all lines is the same, i.e., $r/x = 1/\alpha$ (some constant). In this case, the inverse covariance of voltage magnitude fluctuations is given by ${\Sigma^{LC}_{v}}^{-1} = H_{1/r}{\Sigma^{LC}_{p+\alpha q}}^{-1}H_{1/r}$ (see \cite{bolognani2013identification}). However for non-constant $r/x$ ratio, the graphical model for voltage magnitudes alone may include several additional edges. This is unlike Theorem \ref{structure_LC_PF} where the structure is true even for non-constant $r/x$ values.
\end{itemize}

In the next section, we discuss methods to estimate the inverse covariance matrix and thereby learn the graphical model of voltages in DC-PF or LC-PF.

\section{Estimation of Graphical Model of Voltages}
\label{sec:graphicalmodellearning}
We describe the estimation of the inverse covariance matrix of phase angles ${\Sigma^{DC}_{\theta}}^{-1}$ for the DC-PF with Gaussian injection fluctuations with $0$ mean. The inverse covariance matrix for voltage magnitudes and phase angles for the LC-PF can be estimated using the same techniques. Consider $n$ i.i.d. samples of phase angle vectors $\{\theta^k,1\leq k\leq n\}$ in the grid. Note that if sufficient number of samples (much greater than the number of nodes) are available one can determine ${\Sigma^{DC}_{\theta}}^{-1}$ directly by \textbf{inverting the sample covariance matrix}.

On the other hand, if the number of samples are not large enough (comparable to the number of grid nodes), we consider the maximum likelihood estimator of a Gaussian graphical model \cite{wainwright2008graphical} with a constraint for the number of edges in the grid.
\begin{align}
{\Sigma^{DC}_{\theta}}^{-1} = \arg \min_{S} -&\log\det S +\langle S, \sum_k \frac{\theta^k{\theta^k}^T}{n}\rangle +\lambda\|S\|_1 \label{graphlasso}
\end{align}
Here $\sum_k \theta^k{\theta^k}^T /n$ represents the empirical covariance matrix of phase angles. Further, $|S|_1$ is the $l_1$-norm of the inverse covariance matrix and is a proxy for the $l_0$ norm that measures the sparsity of the grid edges. The optimization problem (\ref{graphlasso}) is convex and termed \textbf{Graphical Lasso} \cite{yuan2007model,tibshirani}. Further, one can use a regression based optimization \cite{lassograph} to determine entries in the inverse covariance matrix for each node $i$ in a distributed fashion. Given the estimate of graphical model for the DC-PF and LC-PF models, we present two methods to distinguish between true topological edges and the `spurious' edges between two-hop neighbors in the grid. The first method uses local counting of nodal neighborhoods in the graphical model to determine the true topological edges and is discussed in the next section.

\section{Topology Learning using Neighborhood Counting}
\label{sec:neigh}
Consider power grid $\cal G$ with at least three non leaf nodes (one path of length least $5$ exists in $\cal G$). Let the graphical model for nodal voltages in DC-PF or LC-PF be known. The objective of topology learning is to identify the true edges between the non-reference buses in the grid. We first consider learning when $\cal G$ is radial.

\subsection{Learning Radial Grids}
Consider the DC-PF model with nodal phase angles in radial grid $\cal G$. The following theorem enables us to distinguish between true and spurious edges in the graphical model.
\begin{theorem}\label{condind}
Let $\cal GM$ be the graphical model for nodal phase angles under the DC-PF model with Gaussian injection fluctuations in radial grid $\cal G$. \\
(a) For edge $(ij)$ in $\cal GM$, there exists nodes $k$ and $l$ separated by $2$ hops in $\cal GM$ with paths $k - i -l$ and $k-j-l$ iff $(ij)$ is a true edge in $\cal G$ between non-leaf nodes $i$ and $j$ \\
(b) Iff $(ki)$ is a true edge in $\cal G$ between leaf node $k$ and non-leaf node $i$, the non-leaf neighbors of $k$ in $\cal GM$ and the non-leaf neighbors of $i$ in $\cal G$ are the same.
\end{theorem}
\begin{proof}
Note that if $k,l$ are two hops away in $\cal GM$, they must be at least three hops away in $\cal G$. For the `If' part of statement (a), consider nodes $k$ and $l$ on either side of edge $(ij)$ in radial tree $\cal G$ as shown in Fig.~\ref{fig:radial1}. They are separated by $2$ hops in $\cal GM$ and connected by two-hop paths $k-i-l$ and $k-j-l$. For the converse, consider $i, j$ as two-hop neighbors with single common neighbor $c$ in $\cal G$. Note that existence of path $k-i-l$ and $k-j-l$ in $\cal GM$ requires $k$ and $l$ to be one and/or two-hop neighbors of both $i$ and $j$ in $\cal G$. As $\cal G$ is radial, this is possible only if $k$ and $l$ include $c$ or any of its immediate neighbors. However, this makes the separation between $k,l$ in $\cal GM$ equal to one hop. Thus no $k,l$ exist that are two hops away in $\cal GM$.

The `if' part of statement (b) follows immediately for each leaf node and its non-leaf parent in $\cal G$ (see Fig.~\ref{fig:radial1}). The converse statement can be proven through contradiction by considering a non-leaf node that is a two-hop neighbor of leaf node $k$ in $\cal G$. (See \cite{sauravacc} for detailed proof of a similar statement as statement (b)).
\end{proof}
Thus the first statement in the theorem enables the discovery of non-leaf nodes and topological edges between them in grid $\cal G$ through counting. The second statement subsequently identifies the edges connected to the leaf nodes leading to exact reconstruction.

\begin{figure}[bt]
\centering\hfill
\subfigure[]{\includegraphics[width=0.15\textwidth]{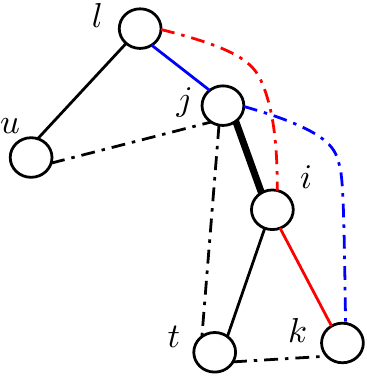}\label{fig:radial1}}\hfill
\subfigure[]{\includegraphics[width=0.15\textwidth]{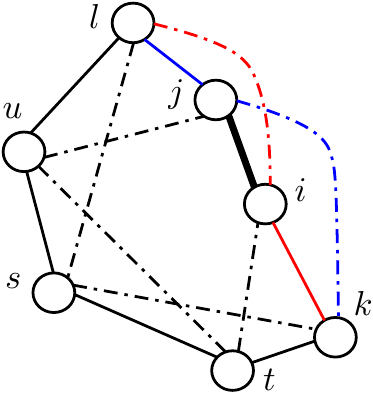}\label{fig:radial2}}\hfill
\vspace{-.25cm}
\caption{Graphical model for phase angles in (a) radial grid (b) loopy grid with cycle length $7$. Solid lines denote true topological edges, dashed lines denote spurious edges in the graphical model. Red edges and blue edges represent two paths between nodes $k$ and $l$ of length $2$ in either graphical model.}
\label{fig:junction2}
\vspace{-3mm}
\end{figure}

\textbf{Note:} We consider the LC-PF model where the graphical model includes nodal voltage magnitudes and phase angles. As mentioned in the previous section, combining nodes pertaining to voltage magnitude and phase angle at the same bus leads to a `hybrid' graph with identical structure as the graphical model for phase angles in DC-PF. Thus, Theorem \ref{condind} can be used to estimate the true edges in radial grid $\cal G$ under $LC-PF$ model as well.

Next, we discuss learning the topology in loopy grids with cycle length greater than $6$.
\subsection{Learning Loopy Grids with cycle length greater than $6$}
The following theorem states that Statement (a) in Theorem \ref{condind} holds for loopy grid $\cal G$ with minimum cycle length greater than $6$ and can be used to distinguish between true and spurious edges.
\begin{theorem}\label{condind_loopy}
Let $\cal GM$ be the graphical model for nodal phase angles under the DC-PF model with Gaussian injection fluctuations in loopy grid $\cal G$ with cycle length greater than $6$. For edge $(ij)$ in $\cal GM$, there exists nodes $k$ and $l$ separated by $2$ hops in $\cal GM$ with paths $k-i-l$ and $k-j-l$ iff $(ij)$ is a true edge in $\cal G$ between non-leaf nodes $i$ and $j$.
\end{theorem}
\begin{proof}
If edge $(ij)$ exists in $\cal G$, consider nodes $k$ and $l$ in cycle $k -i -j -l- r_1- r_2-r_3-..-k$ in $\cal G$ of length $7$ or more. Note that nodes $k,l$ satisfy the condition in the `if' statement for graphical model $\cal GM$ as shown in Fig.~\ref{fig:radial2}. For the converse, consider the case where $i$ and $j$ are not neighbors but two-hop neighbors in $\cal G$. Further, existence of paths $k-i-l$ and $k-j-l$ in $\cal GM$ implies that $k,l$ must be one or two-hop neighbors of both $i$ and $j$ in $\cal G$. From the minimum cycle length constraint, $i$ and $j$ has exactly one common neighbor in $\cal G$, say node $c$. Further, $k,l$ cannot both be neighbors of $i$ or $j$ or $c$ in $\cal G$ as that would make $k,l$ one hop neighbors in $\cal GM$. First, consider the configuration in $\cal G$ where $k$ is neighbor of $i$ and two hops away from $j$. This leads to a cycle $i-c-j-r_1-k-i$ of cycle $5$ and is not permissible. Similarly, node $l$ cannot be neighbor of $i$ and two-hop neighbor of $j$. Next, consider the configuration where $k= c$ is the common neighbor of $i$ and $j$, while $l$ is two hops away from both $i$ and $j$. This configuration produces a cycle $i-k-j-r_1-l-r_2-i$ of length $6$ for some $r_1\neq r_2$ and violates the cycle constraint. Finally consider the case where $k,l$ are two hops away from both $i$ and $j$. This leads to at least one cycle of type $i-c-j-r_1-(k \text{~or~} l)-r_2-i$ of length $6$. Thus no configurations are permissible thereby proving that the converse of the statement is true for grids with minimum cycle length greater than $6$.
\end{proof}
Following the argument after Theorem \ref{condind}, it is clear that Theorem \ref{condind_loopy} also holds for the `hybrid' graph for the LC-PF model where graphical model nodes for votlage magnitude and phase angle at the same bus are combined to form a `hybrid' node. The topology reconstruction steps for grids with minimum cycle length greater than $6$ are described in Algorithm $1$. The tolerance $\tau_1$ is used to determine edges in the graphical model using entries in the estimated inverse covariance matrix of voltages. Note that Algorithm $1$ does not need additional values on nodal injection covariances or line impedances and relies only on nodal voltage samples.
\begin{algorithm}
\caption{Topology Learning using Neighborhood Counting}
\textbf{Input:} Inverse covariance matrix of nodal voltages $\Sigma^{-1}_{V}$: $\Sigma^{-1}_{V} = {\Sigma^{DC}_{\theta}}^{-1}$ for DC-PF or $\Sigma^{-1}_{V} = {\Sigma^{LC}_{(v,\theta)}}^{-1}$ for LC-PF, tolerance $\tau_1 > 0$\\
\textbf{Output:} Grid $\cal G$
\begin{algorithmic}[1]
\State Construct graphical model $\cal GM$ for voltages in DC or LC-PF with edges $(ij)$ for $|\Sigma^{-1}_{V}(ij)| \geq \tau_1$
\If{data from LC-PF}
 \State Construct `hybrid' graph by combining nodes for voltage magnitude and phase at same bus in graphical model
\EndIf
\State Insert edges between non-leaf nodes in $\cal G$ using Theorem \ref{condind_loopy}.
\State Draw edges between leaf nodes and parent in $\cal G$ using Statement (b) in Theorem \ref{condind}.
\end{algorithmic}
\end{algorithm}
In the next section, we discuss another technique to determine the topological edges from the graphical model.

\section{Topology Identification using Thresholding}
\label{sec:thres}
Here we analyze the values in the inverse covariance matrix of voltages in the DC or LC-PF models to determine true edges in grid $\cal G$ from the edges in the graphical model. First we discus the case of radial grids.

\subsection{Learning Radial Grids}
Consider the DC-PF model for nodal phase angles in radial grid. The next result distinguishes between one-hop and two-hop neighbors in the grid graph.
\begin{theorem}\label{threshld_radial}
Let ${\Sigma^{DC}_{\theta}}^{-1}$ be the inverse covariance matrix of nodal phase angles in radial grid $\cal G$ under the DC-PF model with Gaussian nodal injection fluctuations. If ${\Sigma^{DC}_{\theta}}^{-1}(i,j) < 0$, then $(ij)$ is a true edge in $\cal G$
\end{theorem}
\begin{proof}
Observe the expression for ${\Sigma^{DC}_{\theta}}^{-1}(i,j)$ given in Theorem \ref{structure_DC_PF}. For edge $(ij)$, no $k$ exists with edges $(ik), (jk)$ as $\cal G$ is a radial grid. Thus ${\Sigma^{DC}_{\theta}}^{-1}(i,j) <0$ if $(ij)$ is a true edge. Similarly, it can be shown that ${\Sigma^{DC}_{\theta}}^{-1}(i,j)> 0$ if $i$ and $j$ are two-hop neighbors in $\cal G$.
\end{proof}
A similar result for voltage magnitudes in the LC-PF model with constant $r/x$ ratio is given in \cite{bolognani2013identification}. Our next result shows that the inverse covariance matrix of nodal voltage magnitudes and phase angles in the LC-PF model can in fact be used to estimate the radial topology without the restriction on constant $r/x$ ratio.

\begin{theorem}\label{threshld_radial_LC}
Let ${\Sigma^{LC}_{(v,\theta)}}^{-1} = \begin{bmatrix} J_{vv} &J_{v\theta}\\J_{\theta v} &J_{\theta\theta}\end{bmatrix}$ be the inverse covariance matrix of nodal voltages in radial grid $\cal G$ under the LC-PF model with Gaussian nodal injection fluctuations. If $J_{vv}(i,j) + J_{\theta\theta}(i,j) < 0$, then $(ij)$ is a true edge in $\cal G$
\end{theorem}
\begin{proof}
From expressions for ${\Sigma^{LC}_{(v,\theta)}}^{-1}, D$ in Theorem \ref{inverse_LC_PF},
\squeezeup
\begin{align}
 J_{vv} +J_{\theta\theta} &=~\text{\small $H_{g}D^{-1}(\Sigma^{LC}_{qq}+\Sigma^{LC}_{pp})H_{g}$}\nonumber\\
 &~~\text{\small$+H_{\beta}D^{-1}(\Sigma^{LC}_{qq}+\Sigma^{LC}_{pp})H_{\beta}$}
\end{align}
As $D^{-1}(\Sigma^{LC}_{qq}+\Sigma^{LC}_{pp})$ is a diagonal matrix with positive entries, a similar analysis as Theorem \ref{threshld_radial} proves the statement.
\end{proof}
Theorems \ref{threshld_radial} and \ref{threshld_radial_LC} thus provide a simple \textbf{thresholding based} scheme to identify the true edges in radial grid under both DC-PF and LC-PF. Next we show that these results extend to loopy grids without cycles of length $3$.
\subsection{Learning Loopy Grids without Triangles}
A triangle is a sub-graph with three nodes $i,j,k$ and edges $(ij),(ik),(jk)$. In other words, a triangle represents a cycle of length $3$. For neighboring buses $i,j$ in loopy grid $\cal G$ without triangles, there is no $k$ such that $(ki),(kj)$ are edges. Thus the following corollary holds:
\begin{corollary}\label{threshld_triangle}
For loopy grid $\cal G$ with cycle lengths greater than $3$ with Gaussian nodal injection fluctuations, the result in Theorem \ref{threshld_radial} and Theorem \ref{threshld_radial_LC} hold under the DC-PF and LC-PF models respectively.
\end{corollary}
The steps in topology learning for grids without triangles are listed in Algorithm $2$.
\begin{algorithm}
\caption{Topology Learning using Thresholding}
\textbf{Input:} Inverse covariance matrix of nodal voltages, ${\Sigma^{DC}_{\theta}}^{-1}$ for DC-PF or $\Sigma^{-1}_{V} =$ \text{\small$\begin{bmatrix} J_{vv} &J_{v\theta}\\J_{\theta v}&J_{\theta\theta}\end{bmatrix}$} for LC-PF, tolerance $\tau_2 < 0$\\
\textbf{Output:} Grid $\cal G$
\begin{algorithmic}[1]
\ForAll{buses $i,j \in {\cal G}$}
 \State Insert $(ij)$ in $\cal G$ if ${\Sigma^{DC}_{\theta}}^{-1}(i,j) \leq \tau_2$ for DC-PF or $J_{vv}(i,j)+J_{\theta\theta}(i,j)\leq \tau_2$ for LC-PF.
\EndFor
\end{algorithmic}
\end{algorithm}

The tolerance $\tau_2 < 0$ is selected to account for the estimation error in the inverse covariance matrix due to finite number of samples. For exact inverse covariance matrices, the tolerance can be kept at $0$. The estimation steps for the inverse covariance matrix in described in the previous section. Finally we discuss conditions under which Algorithm $2$ is able to learn the topology for grids with triangles.

\subsection{Topology Learning in Loopy Grids with Triangles}
Consider a loopy grid $\cal G$ with multiple triangles (cycles of length $3$). We are interested in understanding the use of Algorithm $2$ in identifying the true topology from the graphical model $\cal GM$. We consider the DC-PF with inverse covariance matrix of phase angles, ${\Sigma^{DC}_{\theta}}^{-1}$. We first analyze the case where an edge in the graph is part of only one triangle (see Fig.~\ref{fig:triangle1}). The next theorem states a sufficient condition under which Algorithm $2$ is guaranteed to recover that edge.
\begin{figure}[bt]
\centering\hfill
\subfigure[]{\includegraphics[width=0.20\textwidth]{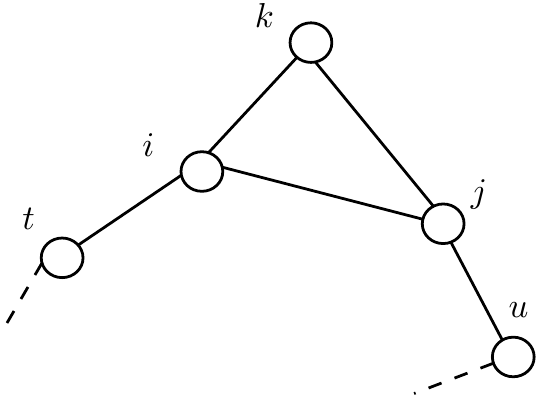}\label{fig:triangle1}}\hfill
\subfigure[]{\includegraphics[width=0.20\textwidth]{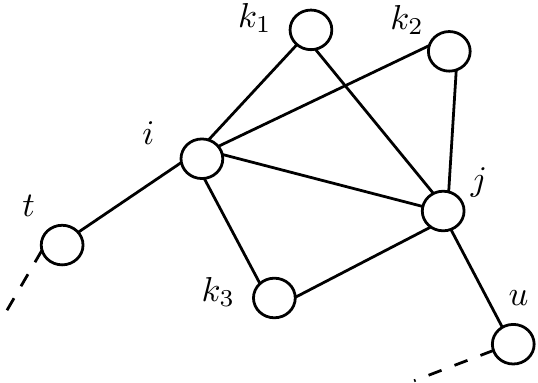}\label{fig:triangle2}}\hfill
\vspace{-.25cm}
\caption{Edge $(ij)$ in loopy grid with (a) one triangle formed with node $k$ (b) multiple triangles formed with set ${\cal K} = \{k_1, k_2, k_3\}$}
\vspace{-3mm}
\end{figure}
\begin{theorem}\label{threshld_loopy}
Let edge $(ij)$ be part of only one triangular sub-graph on nodes $i,j, k$ with edges $(ij), (jk), (ik)$. ${\Sigma^{DC}_{\theta}}^{-1}(i,j) < 0$ in the DC-PF if the following holds:
\begin{align}
&\text{\footnotesize
$\beta_{ij} > -\frac{\Sigma^{DC}_p(j,j)\beta_{ik}+ \Sigma^{DC}_p(i,i)\beta_{jk}}{2(\Sigma^{DC}_p(i,i)+\Sigma^{DC}_p(j,j))} +$}\nonumber\\
&\sqrt{\begin{aligned}&\text{\footnotesize
$\left(\frac{\Sigma^{DC}_{p}(j,j)\beta_{ik}+ \Sigma^{DC}_p(i,i)\beta_{jk}}{2(\Sigma^{DC}_p(i,i)+\Sigma^{DC}_p(j,j))}\right)^2$}\\ &\text{\footnotesize
$+\frac{\Sigma^{DC}_p(i,i)\Sigma^{DC}_p(j,j)\beta_{ik}\beta_{jk}}{\Sigma^{DC}_p(k,k)(\Sigma^{DC}_p(i,i)+\Sigma^{DC}_p(j,j))}$}\end{aligned}}
\end{align}
\end{theorem}
\begin{proof} From Theorem \ref{structure_DC_PF}, ${\Sigma^{DC}_{\theta}}^{-1}(i,j) < 0$ if
\begin{align}
 \frac{\beta_{ik}\beta_{jk}}{\Sigma^{DC}_p(k,k)} < \beta_{ij}(\frac{\beta_{ij}+\beta_{ik}}{\Sigma^{DC}_p(i,i)}+\frac{\beta_{ij}+\beta_{jk}}{\Sigma^{DC}_p(j,j)})\nonumber
\end{align}
Using conditions for positivity of quadratic functions, the statement follows.
\end{proof}
Using similar techniques for edges that are part of multiple triangles (see Fig.~\ref{fig:triangle2}), we have the following sufficiency result.
\begin{theorem}\label{threshld_loopy1}
Consider edge $(ij)$ in graph $\cal G$. Let $\cal K$ be the set of nodes that are neighbors of both $i$ and $j$. ${\Sigma^{DC}_{\theta}}^{-1}(i,j) < 0$ for the DC-PF if the following holds:
\begin{align}
&\text{\footnotesize
$\beta_{ij} >-\frac{\Sigma^{DC}_p(j,j)(\beta_i-\beta_{ij})+ \Sigma^{DC}_p(i,i)(\beta_j-\beta_{ij})}{2({\Sigma^{DC}_p(i,i)+\Sigma^{DC}_p(j,j)})}+$}\nonumber\\
&\sqrt{\begin{aligned}&\text{\footnotesize
$\left(\frac{\Sigma^{DC}_p(j,j)(\beta_i-\beta_{ij})+ \Sigma^{DC}_p(i,i)(\beta_j-\beta_{ij})}{2({\Sigma^{DC}_p(i,i)+\Sigma^{DC}_p(j,j)})}\right)^2$}\\
&\text{\footnotesize
$+\frac{\Sigma^{DC}_p(i,i)\Sigma^{DC}_p(j,j)}{\Sigma^{DC}_p(i,i)+\Sigma^{DC}_p(j,j)}\smashoperator[r]{\sum_{k \in {\cal K}}}\frac{\beta_{ik}\beta_{jk}}{\Sigma^{DC}_p(k,k)}$}\end{aligned}}
\end{align}
where $\beta_i = \smashoperator[r]{\sum_{k:(ik) \in {\cal E}}}\beta_{ik}$.
\end{theorem}
Under Theorem \ref{threshld_loopy1}, all edges in a general loopy graph with multiple triangles can be learned using Algorithm $2$ from the inverse covariance matrix of phase angles in the DC-PF model. Note that this is only a sufficient condition and the grid topology may be learned even if it is violated. Similar relations can also be derived for the LC-PF model, but are omitted for space constraints.

We now consider two interesting cases for the DC-PF model that highlight the rationale behind the previous two theorems. First, we consider the case where covariances of injection fluctuations are equal at all nodes. This may be reasonable in a small distribution grid with comparable nodal injections such that their fluctuations are similar. In this case the following condition ensures correct edge identification using Algorithm $2$.

\begin{theorem}\label{threshld_loopy2}
Consider edge $(ij)$ in graph $\cal G$ where injection covariances at all nodes are equal. Let $\cal K$ with cardinality $|{\cal K}|$ be the set of nodes that are neighbors of both $i$ and $j$. ${\Sigma^{DC}_{\theta}}^{-1}(i,j) < 0$ if the following holds:
\begin{align}
\beta_{ij} > \frac{\max_{k \in {\cal K}, r\in\{i,j\}}\beta_{kr}}{1+\sqrt{1+2/|{\cal K}|}}
\end{align}
\end{theorem}
\begin{proof}
Simplifying the inequality condition in Theorem \ref{threshld_loopy1} under equal injection covariances leads to
\begin{align}
&\beta_{ij}(\beta_{i}+\beta_{j}) > \sum_{k \in {\cal K}}\beta_{ik}\beta_{jk}\nonumber\\
\Rightarrow~& (2+|{\cal K}|)\beta^2_{ij} + \beta_{ij}(\smashoperator[lr]{\sum_{{\small\begin{aligned} k\not\in{\cal K}\\(ik)\in{\cal E}\end{aligned}}}}\beta_{ik}+\smashoperator[lr]{\sum_{{\small\begin{aligned} k\not\in{\cal K}\\(jk)\in{\cal E}\end{aligned}}}}\beta_{jk}) > \sum_{k \in {\cal K}}\begin{aligned}(\beta_{ik} -\beta_{ij})\\(\beta_{jk}-\beta_{ij})\end{aligned} \nonumber
\end{align}
This holds true when $(2+|{\cal K}|)\beta^2_{ij}> |{\cal K}|\smashoperator[r]{\max_{k \in {\cal K}, r\in\{i,j\}}}(\beta_{kr}-\beta_{ij})^2$ which leads to the statement in the theorem.
\end{proof}
We also consider the case where the susceptance per unit length is the same on all lines. This may be true for meshed urban grids where the lines are built using same material. In that setting, we have the following corollary to Theorem \ref{threshld_loopy2} that ensures correct topology learning by Algorithm $2$.
\begin{corollary}\label{threshld_distance}
Consider graph $\cal G$ with equal injection covariances at all nodes and constant susceptance per unit length on all lines. Let $l_{ij}$ be the length of edge $(ij)$ and $\cal K$ be the set of nodes that are neighbors of both $i$ and $j$. ${\Sigma^{DC}_{\theta}}^{-1}(i,j) < 0$ for the DC-PF if the following holds:
\begin{align}
l_{ij} > \frac{\max_{k \in {\cal K}, r \in \{i,j\}}l_{kr}}{1+\sqrt{1+2/|{\cal K}|}}
\end{align}
\end{corollary}
This result follows immediately by using $\beta_{ij} = l_{ij}\beta_u$ where $\beta_u$ is the constant susceptance per unit length on all lines. Note that Corollary \ref{threshld_distance} is a sufficiency condition based on \textbf{the geometry of the grid}. For example, for $\cal K$ of cardinality $1$ and $2$, the right-side in the inequality becomes $\frac{\max_{r \in \{i,j\}}l_{ik}}{1+\sqrt{3}}$ and $\frac{\max_{k \in {\cal K}, r\in\{i,j\}}l_{ik}}{1+\sqrt{2}}$ respectively. The lower limit on $l_{ij}$ thus increases with an increase in the cardinality of ${\cal K}$. Further, it signifies that if line lengths are more equitable, Algorithm $2$ can guarantee recovery in the presence of greater number of triangles in the grid.

Finally for completion, we also consider simultaneous parameter and topology learning in grid $\cal G$ when the observer has access to the covariance of nodal injections in the grid nodes along with the empirical covariance matrix of voltages, created using measurement samples. The covariance of injections may be learned from historical data or other off-line methods.

\begin{theorem}
Consider grid $\cal G$ with covariance matrix of phase angle $\Sigma^{DC}_{\theta}$ and known diagonal injection covariance matrix $\Sigma^{DC}_p$ in the DC-PF model. The edges in the grid are given by the non-zero off-diagonal terms in ${\Sigma^{DC}_p}^{1/2}\sqrt{{\Sigma^{DC}_p}^{-1/2}{\Sigma^{DC}_{\theta}}^{-1}{\Sigma^{DC}_p}^{-1/2}}{\Sigma^{DC}_p}^{1/2}$.
\end{theorem}
\begin{proof}
Note that ${\Sigma^{DC}_p}^{-1/2}{\Sigma^{DC}_{\theta}}^{-1}{\Sigma^{DC}_p}^{-1/2}$ is a positive definite matrix and has a unique square root ${\Sigma^{DC}_p}^{-1/2}H_{\beta}{\Sigma^{DC}_p}^{-1/2}$. Multiplying both side of the square root by ${\Sigma^{DC}_p}^{1/2}$ gives the reduced weighted Laplacian matrix, where the non-zero off-diagonal terms correspond to grid edges and associated susceptance values.
\end{proof}
A similar result can be derived for the LC-PF model as well. It is omitted for brevity. In the next section, we detail numerical simulations on learning the grid topology using our graphical model based framework.

\section{Numerical Simulations}
\label{sec:simulations}
\begin{figure}[bt]
\centering\hfill
\subfigure[]{\includegraphics[width=0.07\textwidth,height =.22\textwidth]{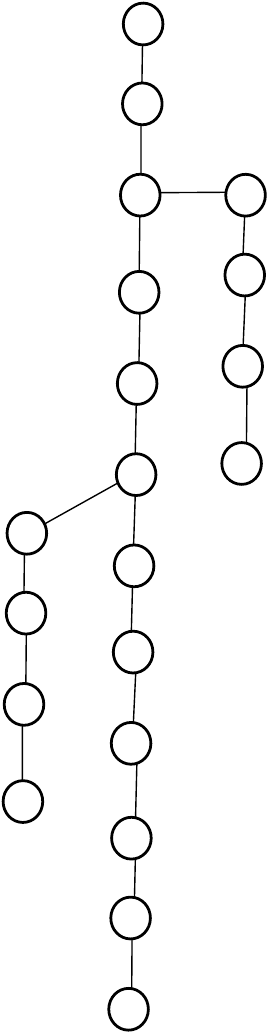}\label{fig:20bus}}\hfill
\subfigure[]{\includegraphics[width=0.07\textwidth,height =.22\textwidth]{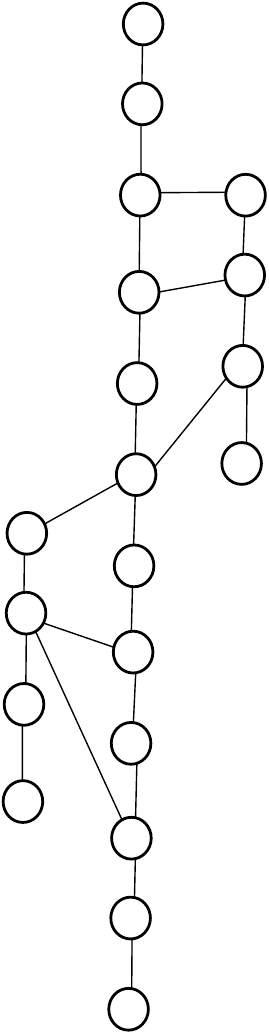}\label{fig:20busb}}\hfill
\subfigure[]{\includegraphics[width=0.07\textwidth,height =.22\textwidth]{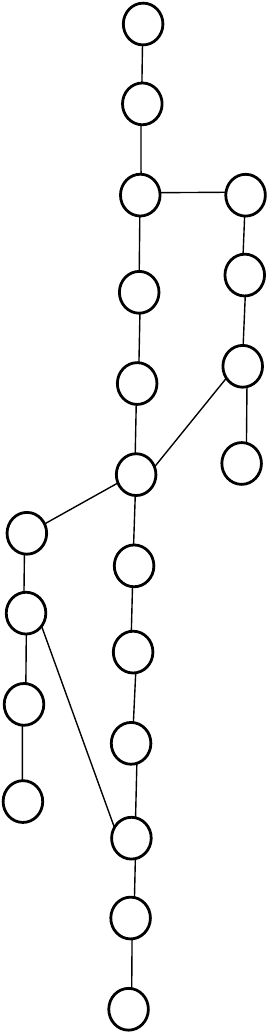}\label{fig:20busc}}\hfill
\vspace{-.25cm}
\caption{(a) $20$ bus radial system (b) $20$ bus loopy grid with minimum cycle length $4$ (c) $20$ bus loopy grid with minimum cycle length $7$}\label{fig:case}
\vspace{-3mm}
\end{figure}
We demonstrate results for Algorithm $1$ (neighborhood counting) and Algorithm $2$ (thresholding) in extracting the operational edge set ${\cal E}$ of power grids using nodal voltage measurements in DC or LC-PF. We consider a $20$ bus radial case \cite{testcase2,dekapscc} and extend it to two loopy grids with minimum cycle length $4$ and $7$ (greater than $3$) as shown in Fig.~\ref{fig:case}. The nodal injection fluctuations are modelled by uncorrelated zero-mean Gaussian random variables. We use DC-PF and LC-PF models to generate i.i.d. samples of nodal voltages and use them to estimate the inverse covariance matrix for the graphical model through Graphical Lasso. The estimated matrix is then input to Algorithms $1$ and $2$ to determine the grid topology. Fig.~\ref{fig:errors_20_neigh} plots the average absolute estimation errors (sum of false positives and false negatives in the estimated topology) in Algorithm $1$ using voltage samples generated by LC-PF model for the radial and loopy test systems in Figs.~\ref{fig:20bus} and \ref{fig:20busc} respectively. Note that the minimum cycle lengths in these two cases is greater than $6$ (sufficient for exact recovery by Algorithm $1$). Thus, increasing the number of samples leads to a reduction in edge detection errors. Next, in Fig.~\ref{fig:errors_20_thres}, we present the performance of Algorithm $2$ in learning the radial and loopy (minimum cycle length $4$) test systems in Figs.~\ref{fig:20bus} and \ref{fig:20busb} using voltage samples in both DC and LC-PF. The performance of Algorithm $2$ outshines that of Algorithm $1$ in terms of number of samples needed.

Finally, we consider the $14$ bus loopy IEEE test case \cite{testsystem} that includes $3$ node cycles as shown in Fig.~\ref{fig:14bus}. The performance of Algorithm $2$ in estimating the true edges using phase angle measurements from DC-PF is shown in Fig.~\ref{fig:error_14_thres}. As learning of loopy grid with cycle length $3$ is not guaranteed to be exact, the average errors decay to a non-zero value for higher sample values. Our ongoing work include efforts to select optimal thresholds in Algorithms $1$ and $2$ to improve their performance.

\begin{figure}[bt]
\centering
\includegraphics[width=0.4\textwidth,height = .34\textwidth]{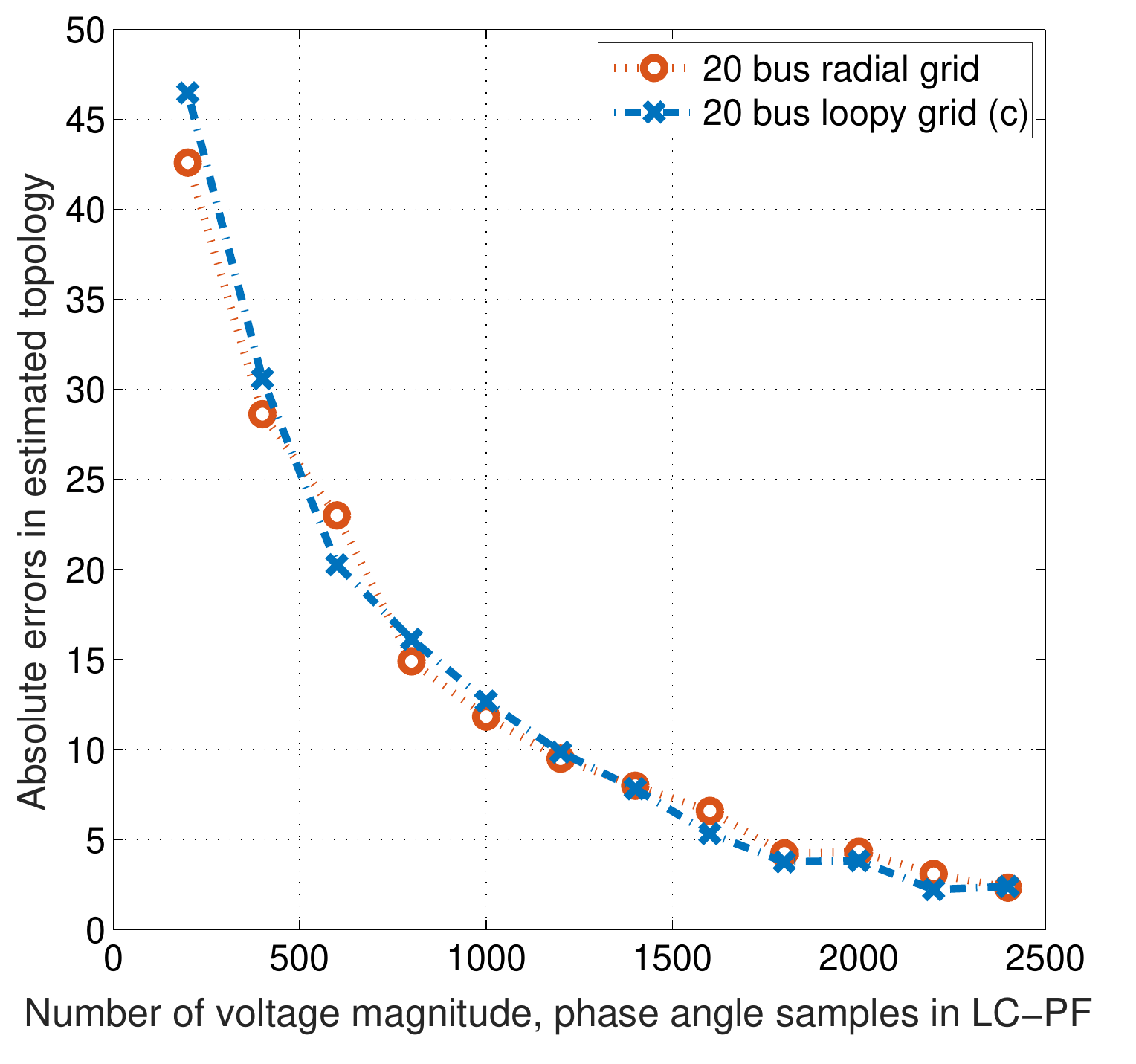}
\caption{Accuracy of Algorithm $1$ with number of voltage measurement samples in LC-PF for radial and loopy test systems in Figs.~\ref{fig:20bus}, \ref{fig:20busc}.}
\label{fig:errors_20_neigh}
\vspace{-3mm}
\end{figure}
\begin{figure}[bt]
\centering
\includegraphics[width=0.4\textwidth,height = .34\textwidth]{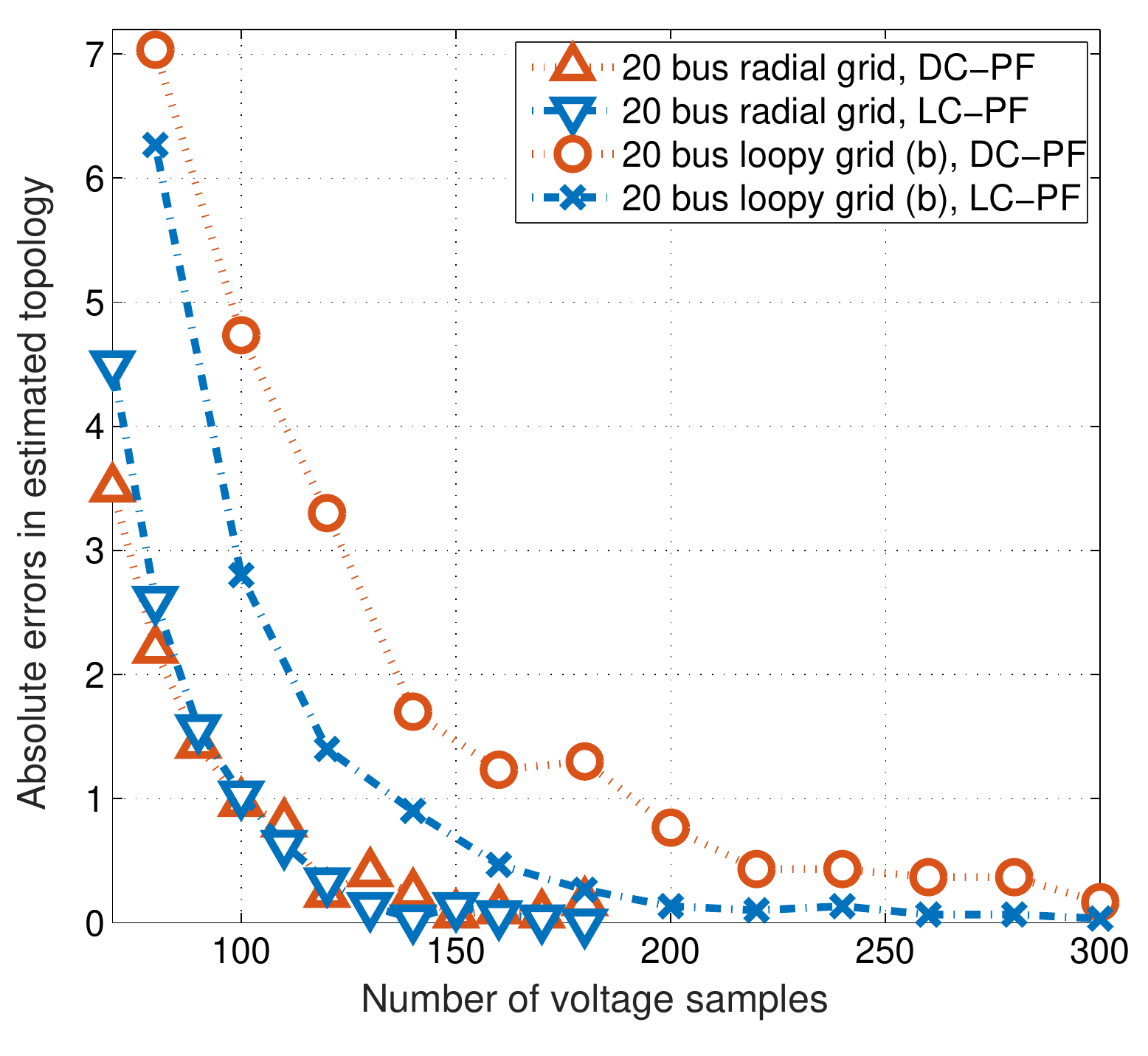}
\caption{Accuracy of Algorithm $2$ with number of voltage measurement samples in DC-PF and LC-PF for radial and loopy test systems in Figs.~\ref{fig:20bus}, \ref{fig:20busb}.}
\label{fig:errors_20_thres}
\vspace{-3mm}
\end{figure}
\label{sec:simulations}
\begin{figure}[bt]
\centering\hfill
\subfigure[]{\includegraphics[width=0.15\textwidth,height =.2\textwidth]{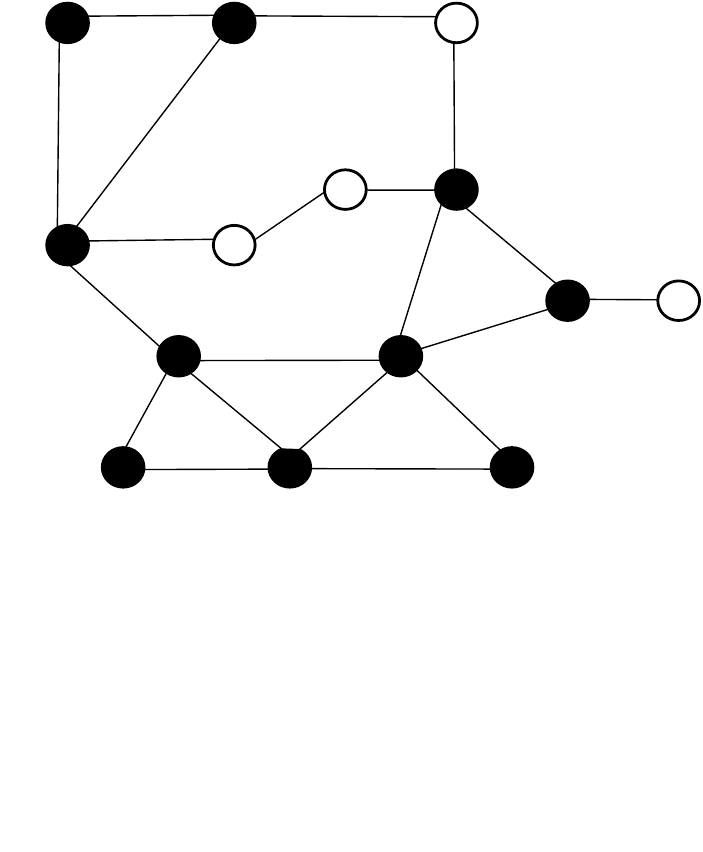}\label{fig:14bus}}\hfill
\subfigure[]{\includegraphics[width=0.33\textwidth,height = .3\textwidth]{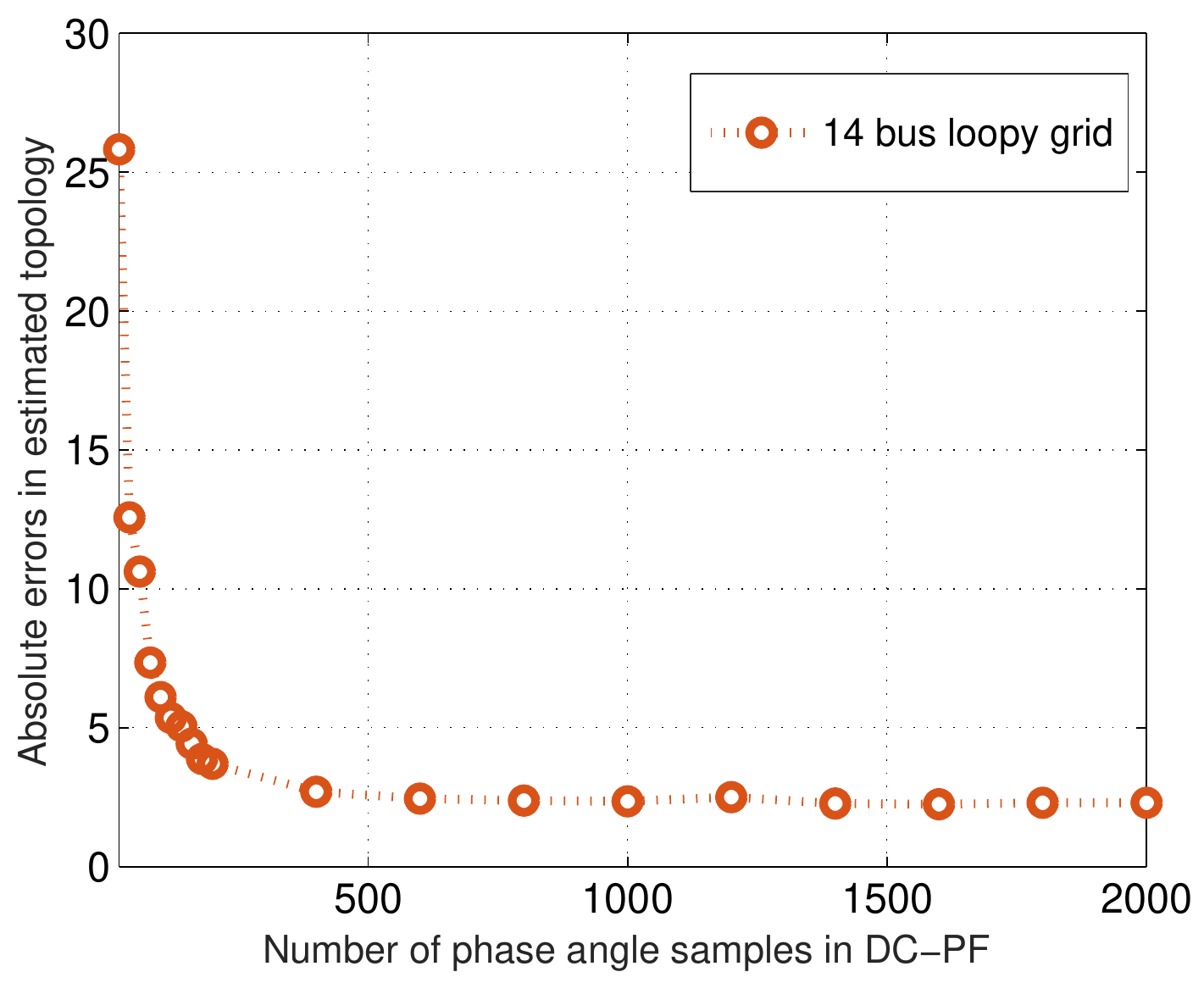}\label{fig:error_14_thres}}\hfill
\vspace{-.25cm}
\caption{(a) $14$-bus loopy IEEE test case. The nodes part of triangles (3 node cycles) are colored solid black. (b) Accuracy of Algorithm $2$ with number of phase angle samples in DC-PF for the $14$ bus system.}
\vspace{-3mm}
\end{figure}

\section{Conclusion}
\label{sec:conclusion}
In this document, we discuss the problem of topology estimation of a loopy power grid from measurements of nodal voltages through a graphical model framework. We demonstrate that the estimated graphical model of voltages for two linear power flow models (DC-PF and LC-PF) includes true edges as well as spurious edges between two-hop neighbors in the power grid. We present two schemes to distinguish between true and fictitious edges. The first algorithm is based on neighborhood counting in the graphical model, while the second algorithm uses a thresholding operator on the inverse covariance matrix of nodal voltages. For loopy grids, we show that the neighborhood counting scheme is able to learn the topology if minimum cycle lengths are greater than six. On the other hand, the thresholding method recovers the correct topology as long as cycle lengths are greater than three. Further, for general loopy grids without any cycle length constraint, we provide sufficient conditions based on injections and line susceptances that enable exact recovery by the thresholding algorithm. We demonstrate the performance of our learning algorithms through numerical simulations. The advantage of our learning framework lies in the fact that it is entirely data driven. It requires only samples of nodal voltage magnitudes and/or phase angles as input and does not require additional information on line parameters and injection statistics.
\bibliographystyle{IEEEtran}
\bibliography{sigproc,FIDVR,SmartGrid,voltage,trees}
\end{document}